\documentclass[a4paper,english,fontsize=11pt,parskip=half,abstract=true]{scrartcl}
\usepackage{babel}
\usepackage[utf8]{inputenc}
\usepackage[T1]{fontenc}
\usepackage[a4paper,left=20mm,right=20mm,top=30mm,bottom=30mm]{geometry}
\usepackage{amsmath}
\usepackage{amsthm}
\usepackage{amssymb}
\usepackage{enumerate}
\usepackage{aliascnt}
\usepackage[bookmarks=true,
            pdftitle={Character tables and defect groups},
            pdfauthor={Benjamin Sambale},
            pdfkeywords={},
            pdfstartview={FitH}]{hyperref}

\newtheorem{Thm}{Theorem} 
\newaliascnt{Lem}{Thm}
\newtheorem{Lem}[Lem]{Lemma}
\aliascntresetthe{Lem}
\newaliascnt{Prop}{Thm}
\newtheorem{Prop}[Prop]{Proposition}
\aliascntresetthe{Prop}
\newaliascnt{Cor}{Thm}

\aliascntresetthe{Cor}
\newaliascnt{Con}{Thm}
\newtheorem{Con}[Con]{Conjecture}
\aliascntresetthe{Con}

\numberwithin{equation}{section}

\allowdisplaybreaks[1]

\renewcommand{\phi}{\varphi}
\newcommand{\C}{\mathrm{C}}
\newcommand{\N}{\mathrm{N}}
\newcommand{\Z}{\mathrm{Z}}
\newcommand{\ZZ}{\mathbb{Z}}
\newcommand{\CC}{\mathbb{C}}
\newcommand{\QQ}{\mathbb{Q}}

\newcommand{\cohom}{\operatorname{H}}
\newcommand{\Aut}{\mathrm{Aut}}

\newcommand{\pcore}{\mathrm{O}}
\newcommand{\GL}{\operatorname{GL}}

\newcommand{\Irr}{\mathrm{Irr}}
\newcommand{\IBr}{\mathrm{IBr}}

\newcommand{\J}{\mathrm{J}}
\newcommand{\Bl}{\operatorname{Bl}}
\newcommand{\Cl}{\mathrm{Cl}}
\newcommand{\Gal}{\operatorname{Gal}}

\newcommand{\Ker}{\operatorname{Ker}}

\newcommand{\sgn}{\operatorname{sgn}}
\newcommand{\diag}{\operatorname{diag}}

\mathchardef\ordinarycolon\mathcode`\:  
 \mathcode`\:=\string"8000
 \begingroup \catcode`\:=\active
   \gdef:{\mathrel{\mathop\ordinarycolon}}
 \endgroup

\title{Character tables and defect groups}
\author{Benjamin Sambale\footnote{Institut für Algebra, Zahlentheorie und Diskrete Mathematik, Leibniz Universität Hannover, Welfengarten 1, 30167~Hannover, Germany,
\href{mailto:sambale@math.uni-hannover.de}{sambale@math.uni-hannover.de}}}
\date{\today}

\begin{document}
\frenchspacing
\maketitle
\begin{abstract}\noindent
Let $B$ be a block of a finite group $G$ with defect group $D$. We prove that the exponent of the center of $D$ is determined by the character table of $G$. In particular, we show that $D$ is cyclic if and only if $B$ contains a “large” family of irreducible $p$-conjugate characters. More generally, for abelian $D$ we obtain an explicit formula for the exponent of $D$ in terms of character values. In small cases even the isomorphism type of $D$ is determined in this situation.
Moreover, it can read off from the character table whether $|D/D'|=4$ where $D'$ denotes the commutator subgroup of $D$. We also propose a new characterization of nilpotent blocks in terms of the character table.
\end{abstract}


\section{Introduction}

A major problem in character theory is to decide which properties of a finite group $G$ can be read off from the complex character table $X(G)$ of $G$. In this note we focus on properties of $p$-blocks of $G$ and their defect groups. For motivational purpose we review some results on the principal $p$-block of $G$ (or any block of maximal defect). It is known that $X(G)$ determines the following properties of a Sylow $p$-subgroup $P$ of $G$:
\begin{enumerate}[(1)]
\item $|P|$ (only the first column of $X(G)$ is needed).

\item whether $P$ is abelian. For $p=2$, this is an elementary result of Camina--Herzog~\cite{CaminaHerzog} (cf. \cite{NavarroTiep}), but it requires the classification of finite simple groups (CFSG for short) if $p$ is odd (see \cite{KimmerleSandling,NST}). If $P$ is abelian, also the isomorphism type of $P$ can be read off from $X(G)$, albeit there is no easy way of doing this (see \cite{KimmerleSandling}). 

\item\label{expP} the exponent of the center $\Z(P)$ (see \cite[Corollary~3.12]{Navarro2}). 

\item whether $P\unlhd G$ (in fact, all normal subgroup orders).

\item whether $P$ has a normal $p$-complement, i.\,e. whether the principal block is nilpotent (only the first column of $X(G)$ is needed, see \cite[Theorem~7.4]{Navarro2}). 

\item\label{B01} whether $\N_G(P)=P\C_G(P)$, i.\,e. whether the principal block has inertial index $1$. This was done by Navarro--Tiep--Vallejo~\cite[Theorem~D]{NTV} for $p>2$ and by Schaeffer-Fry--Taylor~\cite[Theorem~1.7]{SFT} if $p=2$. Both cases rely on the CFSG.

\item\label{NPP} whether $\N_G(P)=P$ (see Navarro--Tiep--Turull~\cite{NTT} for $p>2$ and Schaeffer-Fry~\cite{SF} for $p=2$). Again the CFSG is required.

\item whether $P$ is a TI set. The case $p=2$ appeared in Chillag--Herzog~\cite[Corollary~7]{CH} and the author has verified the result for $p>2$ via the CFSG.

\item the exponent of the abelianization $P/P'$ if $p=2$. This is a special case of a conjecture by Navarro--Tiep~\cite{NTPP'} proved by Malle~\cite{Mallep2} using the CFSG.

\item\label{PP'} whether $|P/P'|=4$ (see \cite{NaSaTi}).
\end{enumerate}

The results on the exponents of $\Z(P)$ and $P/P'$ are of interest, because $X(G)$ does not determine $\exp(P)$ (consider the non-abelian groups of order $p^3$ where $p>2$).

Now let $B$ be an arbitrary $p$-block of $G$ with defect group $D$. 
The distribution of irreducible characters into $p$-blocks is given by $X(G)$ (see \cite[Theorem~3.19]{Navarro}) and the order of $D$ can be computed by the formula
\[|D|=\max\Bigl\{\frac{|G|_p}{\chi(1)_p}:\chi\in\Irr(B)\Bigr\}\]
(here and in the following, $n_p$ and $n_{p'}$ denote the $p$-part and $p'$-part of an integer $n$).
An element $g\in G$ is conjugate to an element of $D$ if and only if $\chi(g)\ne 0$ for some $\chi\in\Irr(B)$ (see \cite[Lemma~22]{HHS}).
In particular, we can decide if $D\unlhd G$. 
Whether or not we can determine if $D$ is abelian would follow from the still unproven Height Zero Conjecture of Richard Brauer.
Recently, Gabriel Navarro has asked me if $X(G)$ determines if $D$ is cyclic. As far as we know this has not yet been observed in the literature (an explicit conjecture for $p\le 3$ appeared in \cite{RSV}). We give an affirmative answer in terms of Galois theory. Recall that $\chi,\psi\in\Irr(G)$ are called \emph{$p$-conjugate} if there exists a Galois automorphism $\gamma$ of $\overline{\QQ}$ such that $\chi^\gamma=\psi$ and $\gamma(\zeta)=\zeta$ for all $p'$-roots of unity $\zeta$ (see next section).

\begin{Thm}\label{cyclic}
Let $B$ be a $p$-block of a finite group $G$ with defect $d>0$. 
Then $B$ has cyclic defect groups if and only if $\Irr(B)$ contains a family of $p$-conjugate characters of size divisible by $p^{d-1}$.
\end{Thm}

Next we show that \eqref{expP} above generalizes to blocks. 
Although this implies \autoref{cyclic}, there is apparently no simple formula to compute $\exp(\Z(D))$ from $X(G)$.

\begin{Thm}\label{expZ}
Let $B$ be a block of a finite group $G$ with defect group $D$. Then the exponent of the center $\Z(D)$ is determined by the character table of $G$.
\end{Thm}

If $D$ is known to be abelian, an explicit formula for $\exp(D)$ can be given in terms of the field of values
\[\QQ(B):=\QQ(\chi(g):\chi\in\Irr(B),\,g\in G).\]
For a positive integer $n$ we denote the $n$-th cyclotomic field by $\QQ_n$. 

\begin{Thm}\label{exp}
Let $B$ be a $p$-block of a finite group $G$ with abelian defect group $D\ne 1$. Let $m:=|G|_{p'}$. Then 
\[\boxed{\exp(D)=p|\QQ(B):\QQ(B)\cap\QQ_m|_p.}\]
If $|D|\le p^5$, then even the isomorphism type of $D$ is determined by the character table.
\end{Thm}

Our last result is a block-wise version of \eqref{PP'}.

\begin{Thm}\label{PP'4}
Let $B$ be a $2$-block of a finite group $G$ with defect group $D$ of order $2^d\ge 8$. Then $|D/D'|=4$ if and only if $|\Irr(B)|<2^d$ and 
\[\QQ(B)\QQ_{|G|_{2'}}\cap\QQ_{2^d}=\QQ(\zeta\pm\zeta^{-1})\]
where $\zeta\in\CC$ is a primitive $2^{d-1}$-th root of unity.
In particular, the character table of $G$ determines if $B$ has tame representation type.
\end{Thm}

Recall that a block $B$ with abelian defect group is nilpotent if and only if $B$ has inertial index $1$.
By work of Okuyama--Tsushima~\cite[Proposition~1 and Theorem~3]{OkuyamaTsushima}, $B$ is nilpotent with abelian defect group if and only if all characters in $\Irr(B)$ have the same degree. More generally, it has been conjectured (and verified in many cases) by Malle--Navarro~\cite{MNchardeg} that $B$ is nilpotent if and only if all height zero characters have the same degree (an invariant of $X(G)$). A different characterization of nilpotent blocks in terms of the focal subgroup was proved by Kessar--Linckelmann--Navarro~\cite{KLN}. It is however not clear if the (order of the) focal subgroup is encoded in $X(G)$. The same remark applies to another conjectural characterization by Puig in terms of counting Brauer characters in Brauer correspondents (see \cite[Conjecture~6.3.3]{CravenGuide} and \cite{WatanabeNil}). In the last section of this paper we propose a strengthening of Puig's Conjecture characterizing nilpotent blocks by a single invariant which is derived from lower defect groups and can be computed from $X(G)$ (see \autoref{con}). 

Note that \eqref{NPP} above does not admit a direct analog for non-principal blocks by Brauer's third main theorem (when $\N_G(P)$ is replaced by the inertial group of a Brauer correspondent). 

\section{Proofs}

Our notation is fairly standard and follows \cite{Navarro}. As usual, we set $k(B):=|\Irr(B)|$ and $l(B):=|\IBr(B)|$ for every block $B$ of a finite group $G$. The generalized decomposition matrix $Q=(d_{\chi\phi}^x)$ of $B$ has size $k(B)\times k(B)$ and entries in the cyclotomic field $\QQ_{\exp(D)}$ where $D$ is a defect group of $B$. The rows of $Q$ are indexed by $\chi\in\Irr(B)$ and the columns are indexed by pairs $(x,\phi)$ where $x\in D$ and $\phi\in\IBr(b)$ for some Brauer correspondent $b$ of $B$ in $\C_G(x)$. Let $\mathcal{G}$ be the Galois group of $\QQ_{|G|}$ with fixed field $\QQ_{|G|_{p'}}$. Characters in the same $\mathcal{G}$-orbit are called $p$-conjugate. Characters fixed by $\mathcal{G}$ are called $p$-rational.
We make use of the natural isomorphisms
\[\mathcal{G}\cong\Gal(\QQ_{|G|_p}|\QQ)\cong(\ZZ/|G|_p\ZZ)^\times.\]
In this way $\mathcal{G}$ acts on the rows and columns of $Q$ via
\[d_{\chi^\gamma,\phi}^x=\gamma(d^x_{\chi\phi})=d^{x^\gamma}_{\chi\phi}\qquad(\gamma\in\mathcal{G}).\]

We recall that the characters of a nilpotent block $B$ with defect group $D$ were parameterized by Broué--Puig~\cite{BrouePuig} using the so-called $*$-construction. More precisely, there exists a $p$-rational character $\chi\in\Irr(B)$ of height $0$ such that $\Irr(B)=\{\lambda*\chi:\lambda\in\Irr(D)\}$. For $\gamma\in\mathcal{G}$, we have $(\lambda*\chi)^\gamma=(\lambda^\gamma)*\chi$.

\begin{proof}[Proof of \autoref{cyclic}]
Let $D$ be a defect group of $B$. If $D$ is not cyclic, then the generalized decomposition matrix $Q$ has entries in $\QQ_{p^{d-1}}$. Hence, the lengths of the $\mathcal{G}$-orbits on the rows of $Q$ divide $\phi(p^{d-1})=p^{d-2}(p-1)$. So there is no family of $p$-conjugate characters in $\Irr(B)$ of size divisible by $p^{d-1}$.

Now suppose that $D=\langle x\rangle$ is cyclic. If $p=2$, then $B$ is nilpotent, because the inertial index of $B$ is $1$. By Broué--Puig~\cite{BrouePuig}, we have $\Irr(B)=\{\lambda*\chi:\lambda\in\Irr(D)\}$ for some $2$-rational $\chi$.
Let $\lambda_1,\ldots,\lambda_{2^{d-1}}$ be the faithful characters of $\Irr(D)$. Then $\lambda_1*\chi,\ldots,\lambda_{2^{d-1}}*\chi$ is a family of $2$-conjugate characters of $B$ of size $2^{d-1}$.
Finally let $p>2$. Then $\mathcal{G}\cong (\ZZ/|G|_p\ZZ)^\times$ is cyclic and the rows and columns of $Q$ form isomorphic $\mathcal{G}$-sets by Brauer's permutation lemma (see \cite[Lemma~IV.6.10]{Feit}). Let $b$ be a Brauer correspondent of $B$ in $\C_G(D)$. 
For $u\in D\setminus\{1\}$ the Brauer correspondents $b_u:=b^{\C_G(u)}$ are nilpotent. In particular, $l(b_u)=1$ and every such $u$ labels a unique column of $Q$. Two elements $u,v\in D$ determine the same column of $Q$ if and only if they are conjugate under the inertial quotient $N:=\N_G(D,b)/\C_G(D)$. We regard $N$ as a $p'$-subgroup of $\Aut(D)$. Since all generators of $D$ are conjugate under $\mathcal{G}$, the $\mathcal{G}$-orbit of the column of $Q$ labeled by $x$ has size $|\Aut(D):N|\equiv 0\pmod{p^{d-1}}$. The corresponding orbit on the rows of $Q$ yields the desired family of $p$-conjugate characters of $\Irr(B)$
\end{proof}

We remind the reader that every $x\in G$ can be written uniquely as $x=x_px_{p'}=x_{p'}x_p$ where the $p$-factor $x_p$ is a $p$-element and the $p'$-factor $x_{p'}$ is a $p'$-element. The $p$-section of $x$ is the set of elements $y\in G$ such that $x_p$ and $y_p$ are conjugate. 

In the following we work over a “large enough“ complete discrete valuation ring $\mathcal{O}$ such that the residue field $\mathcal{O}/\J(\mathcal{O})$ is algebraically closed of characteristic $p$.
The remaining theorems are based on the following observation.

\begin{Prop}\label{major}
Let $B$ be a $p$-block of $G$ with defect group $D$. 
For a given $p$-element $x\in G$, the character table determines the number of Brauer correspondents of $B$ in $\C_G(x)$ with defect group $D$.
\end{Prop}
\begin{proof}
We assume that the column of the character table $X=X(G)$ corresponding to $x$ is given.
Let $q\ne p$ be another prime. By \cite[Theorem~7.16]{Navarro2}, we find all elements $g\in G$ such that the $q'$-factor of $g$ is conjugate to $x$. By induction on the number of prime divisors of the order of an element, the whole $p$-section $S$ of $x$ can be spotted in $X$.
Let $y_1,\ldots,y_l\in\C_G(x)$ be representatives for the conjugacy classes of $p$-regular elements in $\C_G(x)$. Then the elements $xy_1,\ldots,xy_l$ represent the conjugacy classes inside $S$ (see \cite[p.~105]{Navarro}). Let $\IBr(\C_G(x))=\{\phi_1,\ldots,\phi_l\}$.
We construct the matrices 
\begin{align}
X_x&:=\bigl(\chi(xy_i):\chi\in\Irr(B),1\le i\le l\bigr),\nonumber\\
Q_x&:=\bigl(d^x_{\chi\phi_i}:\chi\in\Irr(B),1\le i\le l\bigr),\label{Qx}\\
Y_x&:=\bigl(\phi_i(y_j):1\le i,j\le l\bigr)\nonumber
\end{align}
and observe that $X_x=Q_xY_x$ can be read off of $X$ (see \cite[Corollary~5.8]{Navarro}). Let $b_1,\ldots,b_s$ be the Brauer correspondents of $B$ in $\C_G(x)$. Let $C_i$ be the Cartan matrix of $b_i$ for $i=1,\ldots,s$. Finally, let
\[C_x:=\begin{pmatrix}
C_1&&&0\\
&\ddots\\
&&C_s&\\
0&&&0
\end{pmatrix}\in\ZZ^{l\times l}.\]
Brauer's second main theorem yields
\[X_x^\text{t}\overline{X_x}=Y_x^\text{t}Q_x^\text{t}\overline{Q_xY_x}=Y_x^\text{t}C_x\overline{Y_x}\]
where $X_x^\text{t}$ denotes the transpose and $\overline{X_x}$ the complex conjugate of $X_x$ (see \cite[Lemma~5.13]{Navarro}). 
We may assume that the entries of $X_x,Q_x,Y_x$ lie in the valuation ring $\mathcal{O}$ (recall that these entries are algebraic integers). 
It follows from \cite[Lemma~2.4 and Theorem~1.19]{Navarro} that $Y_x$ is invertible over $\mathcal{O}$. In particular, $X_x^\text{t}\overline{X_x}$ and $C_x$ have the same elementary divisors up to multiplication with units in $\mathcal{O}$ (recall that $\mathcal{O}$ is indeed a principal ideal domain). The largest elementary divisor of $C_i$ is the order of a defect group $D_i$ of $b_i$ and occurs with multiplicity $1$ in $C_i$ (see \cite[Theorem~3.26]{Navarro}). Since $D_i$ is conjugate to a subgroup of $D$, all non-zero elementary divisors of $C_x$ are divisors of $|D|$. Moreover, the number of blocks $b_i$ with defect group $D$ is just the multiplicity of $|D|$ as an elementary divisor of $C_x$.
\end{proof}

In the situation of \autoref{major}, the pairs $(x,b)$ are called ($B$-)\emph{subsections} if $b$ is a Brauer correspondent of $B$ in $\C_G(x)$. The subsection is called \emph{major} if $b$ and $B$ have the same defect (group). One might wonder if the total number of subsections (major or not) can be deduced from the block decomposition of the hermitian matrix $X_x^\mathrm{t}\overline{X_x}$. However, if $B$ is the only block of $G$, then $X_x^\mathrm{t}\overline{X_x}$ is just a diagonal matrix by the second orthogonality relation.

\begin{proof}[Proof of \autoref{expZ}]
The columns of the character table $X=X(G)$ corresponding to $p$-elements are determined via \cite[Corollary~7.17]{Navarro2}.
For a $p$-element $x\in G$, we can decide from $X$ whether there are major subsections $(x,b)$ by \autoref{major}. 
This happens if and only if $x$ is conjugate to some element of $\Z(D)$ (see \cite[Problem~9.6]{Navarro}).
Thus, suppose that $x\in\Z(D)$ has order $p^e$ and $b$ has defect group $D$. Then the matrix $Q_x$ defined in \eqref{Qx} of the previous proof has entries in $\QQ_{p^e}$. The entries of $X_x=Q_xY_x$ generate a subfield $\QQ(X_x)\subseteq\QQ_{p^em}$ where $m:=|G|_{p'}$.
Let $b_D$ be a Brauer correspondent of $b$ (and of $B$) in $D\C_G(D)$ such that $b=b_D^{\C_G(x)}$. In the following we replace $\mathcal{G}$ by the (smaller) Galois group of $\QQ_{p^em}$ with fixed field $\QQ_m$. 
Let $\gamma\in\mathcal{G}$ be a non-trivial $p$-element. By a fusion argument of Burnside, the $B$-subsections $(x,b)$ and $(x^\gamma,b)$ are conjugate in $G$ if any only if $x$ and $x^\gamma$ are conjugate in the inertial group $\N_G(D,b_D)$ (see \cite[Problem~9.7]{Navarro}). Since $x\in\Z(D)$ and $\N_G(D,b_D)/D\C_G(D)$ is a $p'$-group, this cannot happen. Hence, there exist $\chi\in\Irr(B)$ and $\phi\in\IBr(b)$ such that
\[d_{\chi^\gamma,\phi}^x=\gamma(d^x_{\chi\phi})=d^{x^\gamma}_{\chi\phi}\ne d^x_{\chi\phi}\]
and $\chi^\gamma\ne\chi$.  
This shows that $\QQ(X_x)$ does not lie in the fixed field of any non-trivial $p$-element of $\mathcal{G}$. 
Hence by Galois theory, $|\QQ_{p^em}:\QQ(X_x)\QQ_m|$ is a $p'$-number and
\[|\QQ(X_x)\QQ_m:\QQ_m|_p=|\QQ_{p^em}:\QQ_m|_p=|\mathcal{G}|_p=p^{e-1}.\]
Therefore, $X$ determines the order of every $x\in\Z(D)$. In particular, $\exp(\Z(D))$ is determined.
\end{proof}

By the proof above, the character table determines whether all $x\in D$ are conjugate to elements of $\Z(D)$. This is a necessary (but insufficient) criterion for $D$ to be abelian.
Next we prove the first part of \autoref{exp}.

\begin{Prop}
Let $B$ be a $p$-block of $G$ with abelian defect group $D$. 
Let $m:=|G|_{p'}$. Then 
\[\exp(D)=p|\QQ(B):\QQ(B)\cap\QQ_m|_p.\]
\end{Prop}
\begin{proof}
Since $D$ is abelian, all $B$-subsections are major (see \cite[Problem~9.6]{Navarro}). Hence, in the proof of \autoref{expZ} there is no need to consider only one $p$-section at a time. In the end, we can replace $\QQ(X_x)$ by $\QQ(B)$ to obtain
\[p|\QQ(B):\QQ(B)\cap\QQ_m|_p=p|\QQ(B)\QQ_m:\QQ_m|_p=\exp(D).\qedhere\]
\end{proof}

Now we come to the second part of \autoref{exp}.

\begin{Prop}
Let $B$ be a block of $G$ with abelian defect group $D$ and defect at most $5$. Then $X(G)$ determines the isomorphism type of $D$.
\end{Prop}
\begin{proof}
Since $|D|$ and $\exp(D)$ are determined by the character table, we may assume that $|D|\in\{p^4,p^5\}$. Let $T$ be the inertial group of some Brauer correspondent of $B$ in $\C_G(D)$. Since $D$ is abelian, the $G$-conjugacy classes of $B$-subsections correspond to the $T$-orbits on $D$ by \cite[Problems~9.6 and 9.7]{Navarro}. 
For a fixed $x\in D$, the $B$-subsections of the form $(x,b)$ are pairwise non-conjugate since the blocks are ideals of the group algebra of $\C_G(x)$. Since all $B$-subsections are major, \autoref{major} allows us to count the number of subsections (up to conjugation) corresponding to elements $x\in D$ of some fixed order (note that $|\langle x\rangle|$ is determined by the proof of \autoref{expZ}). Hence, the character table determines the number of $T$-orbits on $D$ of elements of order $p^i$ for each $i\ge 0$. 

We call $x,y\in D$ \emph{equivalent} if there exist $t\in T$ and $k\in\ZZ$ such that $x^t=y^{1+kp}$. Equivalent elements clearly have the same order. Let $d_i$ be the number of equivalence classes of elements in $D$ of order $p^i$ for $i\ge 1$. 
Since $T$ acts coprimely on $D$, the distinct elements of the form $x^{1+kp}$ lie in distinct $T$-orbits. Hence, the number of $T$-orbits of elements of order $p^i$ is $d_ip^{i-1}$. In particular, the numbers $d_i$ are determined by $X(G)$.
Note that $d_1$ is just the number of $T$-orbits of elements of order $p$. 

Now we assume that $|D|=p^4$. It suffices to distinguish $D\cong C_{p^2}^2$ from $D\cong C_{p^2}\times C_p^2$. 
Suppose first that $D\cong C_{p^2}^2$. Then every element of order $p$ in $D$ is a $p$-power of some element in $D$. Moreover, if $x,y\in D$ are equivalent, so are $x^p$ and $y^p$. This shows that $d_1\le d_2$.
Next we consider $D=D_1\times D_2\cong C_{p^2}\times C_p^2$. Since $T$ acts coprimely on $D$, we may assume that $D_1\cong C_{p^2}$ and $D_2\cong C_p^2$ are $T$-invariant (see \cite[Theorem~5.2.2]{Gorenstein}). Let $x_1,y_1\in D_1$ be of order $p^2$. Since $D_1$ is cyclic, we see that $x_1$ and $y_1$ are equivalent if and only if $x_1^p$ and $y_1^p$ are equivalent. For any $x_2,y_2\in D_2$, it follows that $x_1x_2$ and $y_1y_2$ are equivalent if and only if $x_1^px_2$ and $y_1^py_2$ are equivalent. Every element of order $p^2$ has the form $x_1x_2$, but the elements of $D_2$ do not have the form $x_1^px_2$. Consequently, $d_1>d_2$.

It remains to discuss the case $|D|=p^5$. If $\exp(D)=p^3$, then we need to distinguish $C_{p^3}\times C_{p^2}$ from $C_{p^3}\times C_p^2$. But this follows immediately from the case $|D|=p^4$ above by considering only elements of order at most $p^2$. Hence, we may assume that $\exp(D)=p^2$. If $D\cong C_{p^2}\times C_p^3$, we obtain $d_1>d_2$ just as in the case $C_{p^2}\times C_p^2$. Finally, let $D=D_1\times D_2\cong C_{p^2}^2\times C_p$ with $T$-invariant subgroups $D_1$ and $D_2$. 
Let $\Delta\subseteq D_1$ be a $T$-orbit of elements of order $p$. Let $\widehat{\Delta}:=\{y\in D_1:y^p\in\Delta\}$. Note that $\widehat{\Delta}$ is a union of equivalence classes and $|\widehat{\Delta}|=p^2|\Delta|$. Since the size of an equivalence class cannot be divisible by $p^2$, $\widehat{\Delta}$ contains at least two equivalence classes. For $x\in\Delta$ we pick non-equivalent elements $\widehat{x},\widetilde{x}\in\widehat{\Delta}$. 
Let $z_0=1,z_1,\ldots,z_s$ be representatives for the $T$-orbits in $D_2$. Let $x,y\in D_1$ be of order $p$. Since $xz_0,\ldots,xz_s$ lie in distinct $T$-orbits, we obtain that $d_1-s>s$. Moreover, if $xz_i$ and $yz_j$ are not equivalent, then $\widehat{x}z_i$, $\widetilde{x}z_i$, $\widehat{y}z_j$ and $\widetilde{y}z_j$ are pairwise non-equivalent elements of order $p^2$. Since every element of order $p$ outside $D_2$ is equivalent to some $xz_i$, it follows that $d_1<2(d_1-s)\le d_2$.
\end{proof}

We do not know if our method extends to blocks of defect $6$, but it definitely does not work for defect $7$. 
In fact, the defect groups $C_4^3\times C_2$ and $C_4^2\times C_2^3$ cannot be distinguished by counting orbits of the inertial quotient $C_7\rtimes C_3$ (given a suitable action, there are three orbits of involutions and eight orbits of elements of order $4$ in both cases). Nevertheless, these groups can still be distinguished by other means.

Finally we prove our last theorem.

\begin{proof}[Proof of \autoref{PP'4}]
Suppose first that $|D/D'|=4$. Then $D$ is a dihedral, a semidihedral or a (generalized) quaternion group by a theorem of Taussky (see \cite[Satz~III.11.9]{Huppert}). It was shown by Brauer and Olsson that $k(B)<2^d$ (see \cite[Theorem~8.1]{habil}). They have also computed the generalized decomposition numbers of $B$, but we only need a small portion of those. For that, let $x\in D$ be of order $2^{d-1}$ and let $b_x$ be a Brauer correspondent of $B$ in $\C_G(x)$. Then $b_x$ is a block with cyclic defect group $\langle x\rangle$. In particular, $b_x$ is nilpotent and $\IBr(b_x)=\{\phi\}$. If $D$ is a dihedral or a quaternion group, then $x$ is conjugate to $x^{-1}$ in $D$. From the structure of the fusion system of $B$ (see \cite[Theorem~8.1]{habil}) we see that there is no more fusion inside $\langle x\rangle$. It follows as in \autoref{expZ} that $\QQ(d^x_{\chi\phi}:\chi\in\Irr(B))=\QQ(\zeta+\zeta^{-1})$. In the semidihedral case we obtain similarly that $\QQ(d^x_{\chi\phi}:\chi\in\Irr(B))=\QQ(\zeta-\zeta^{-1})$.

Next let $y\in D$ be arbitrary. If $y$ has order at most $2$, then the generalized decomposition numbers with respect to $y$ are rational integers. Thus, we may assume that $|\langle y\rangle|>2$. Then $y$ and $y^{-1}$ (or $y^{-1+2^{d-2}}$ if $D$ is semidihedral) are conjugate in $D$ and the Brauer correspondent $b_y$ has cyclic defect group (namely $\langle y\rangle$ or $\langle x\rangle$; see \cite[Lemma~1.34]{habil}). Hence, $\IBr(b_y)=\{\mu\}$ and the argument above yields $d^y_{\chi\mu}\in\QQ(\zeta\pm\zeta^{-1})$ for every $\chi\in\Irr(B)$. 
Therefore, the entries of the generalized decomposition matrix $Q$ of $B$ generate the field $\QQ(Q)=\QQ(\zeta\pm\zeta^{-1})\subseteq\QQ_{2^d}$. 
Let $m:=|G|_{2'}$ and $\gamma\in\Gal(\QQ_{2^d}|\QQ)$. Let $\widehat\gamma$ be the unique extension of $\gamma$ to $\Gal(\QQ_{2^dm}|\QQ_m)$. Then 
\begin{align*}
\gamma\in\Gal(\QQ_{2^d}|\QQ(Q))&\Longleftrightarrow\widehat\gamma\in\Gal(\QQ_{2^dm}|\QQ(Q)\QQ_m)=\Gal(\QQ_{2^dm}|\QQ(B)\QQ_m)\\
&\Longleftrightarrow\gamma\in\Gal(\QQ_{2^d}|\QQ(B)\QQ_m\cap\QQ_{2^d})
\end{align*}
(this argument is due to Reynolds~\cite{ReynoldsField}). The main theorem of Galois theory implies $\QQ(B)\QQ_m\cap\QQ_{2^d}=\QQ(Q)=\QQ(\zeta\pm\zeta^{-1})$ as desired. 

Now assume conversely that $k(B)<2^d$ and $\QQ(B)\QQ_m\cap\QQ_{2^d}=\QQ(\zeta\pm\zeta^{-1})$. 
If $D$ is cyclic or of type $C_{2^{d-1}}\times C_2$ with $d\ge 3$, then $B$ is nilpotent in contradiction to $k(B)<2^d$.
Suppose that $\exp(D)<2^{d-1}$. Then the generalized decomposition numbers of $B$ lie in $\QQ_{2^{d-2}}$ and we obtain 
\[\zeta\pm\zeta^{-1}\in\QQ(B)\QQ_m\cap\QQ_{2^d}\subseteq\QQ_{2^{d-2}m}\cap\QQ_{2^d}=\QQ_{2^{d-2}}.\] 
This forces $d=3$ and $\exp(D)=2$. Then however, $D$ is elementary abelian and $k(B)=8=2^d$ by Kessar--Koshitani--Linckelmann~\cite{KKL}. This contradiction shows that $\exp(D)=2^{d-1}$. Now it is well-known that $|D:D'|=4$ unless $d>3$ and
\[D=\langle x,y:x^{2^{d-1}}=y^2=1,\ yxy^{-1}=x^{1+2^{d-2}}\rangle.\]
In this exception, $B$ is nilpotent by \cite[Theorem~8.1]{habil}. By Broué--Puig~\cite{BrouePuig}, there exists a $2$-rational character $\chi\in\Irr(B)$ such that $\Irr(B)=\{\lambda*\chi:\lambda\in\Irr(D)\}$. This yields the contradiction
\[\QQ(B)\QQ_m=\QQ(D)\QQ_m=\QQ_{2^{d-2}m}.\]
For the last claim recall that $B$ has tame representation type if and only if $D$ is a Klein four-group (detectable by \autoref{exp}) or $D$ is non-abelian and $|D/D'|=4$. 
\end{proof}

We remark that the distinction of the defect groups of order $8$ in the proof above relies implicitly on the classification of finite simple groups (via \cite{KKL}). The dependence on the CFSG can be avoided by making use of the remark after the proof of \autoref{expZ}.
As in \cite{NaSaTi}, the Alperin--McKay Conjecture would imply that $|D/D'|=4$ if and only if $B$ has exactly four irreducible characters of height $0$ (provided $p=2$).

\section{A characterization of nilpotent blocks}

As before, let $B$ be a $p$-block of $G$ with defect group $D$. Let $X(B)$ be the submatrix of $X(G)$ with rows indexed by $\Irr(B)$. By the block orthogonality relation (see \cite[Corollary~5.11]{Navarro}), the matrix $X(B)^\mathrm{t}\overline{X(B)}$ has block diagonal shape. The blocks (of that matrix) are the matrices $X_x^\mathrm{t}\overline{X_x}$ studied in the proof of \autoref{major}. 
Let $(1,B)=(x_1,b_1),\ldots,(x_s,b_s)$ be representatives for the $G$-conjugacy classes of $B$-subsections. 
Let $C_i$ be the Cartan matrix of $b_i$. We have seen in the proof of \autoref{major} that $X(B)^\mathrm{t}\overline{X(B)}$ and the block diagonal matrix $C_1\oplus\ldots\oplus C_s$ have the same \emph{non-zero} elementary divisors $e_1,\ldots,e_k$ over $\mathcal{O}$ (up to multiplication with units in $\mathcal{O}$). Hence, we may assume that $e_1,\ldots,e_k$ are uniquely determined integer $p$-powers. 
We call $e_1,\ldots,e_k$ the \emph{elementary divisors} of $B$.
It turns out that these numbers are the orders of the \emph{lower defect groups} of $B$ (with multiplicities) introduced by Brauer~\cite{BrauerLDG} (see \autoref{blcksp} below). 
We call
\[\boxed{\gamma(B):=\frac{1}{e_1}+\ldots+\frac{1}{e_k}\in\QQ}\]
the \emph{fusion number} of $B$. This definition is inspired by the class equation in finite groups as will become clear in the sequel.

\begin{Con}\label{con}
For every block $B$ of $G$ we have $\gamma(B)\ge 1$ with equality if and only if $B$ is nilpotent.
\end{Con}

In contrast to the character degrees considered in \cite{MNchardeg}, we will see that the fusion number is invariant under categorical equivalences like isotypies. 

To verify that $\gamma(B)\ge1$, it is often enough to consider only the Cartan matrix $C_1$ of $B$. If $C_1$ possesses an entry coprime to $p$, then $1$ is an elementary divisor and $\gamma(B)\ge1$ with equality if and only if $B$ has defect $0$. 
The remaining elementary divisors $e_i>1$ can in principle be computed locally (see \cite[Theorem~4.3]{OlssonLDG}).

Before providing evidence for \autoref{con}, we derive a consequence which strengthens Puig's Conjecture (mentioned in the introduction) and was established for abelian defect groups in \cite{PuigWatanabe}.

\begin{Prop}\label{puig}
\autoref{con} implies that $B$ is nilpotent if and only if $l(b)=1$ for every $B$-subsection $(x,b)$. 
\end{Prop}
\begin{proof}
It is well-known that every nilpotent block fulfills the condition. Suppose conversely that $l(b)=1$ for every $B$-subsection $(x,b)$. Let $D$ be a defect group of $B$ and let $b_D$ be a Brauer correspondent of $B$ in $D\C_G(D)$. By \cite[Lemma~1.34]{habil}, there exist representatives $(x_1,b_1),\ldots,(x_s,b_s)\in(D,b_D)$ for the $G$-conjugacy classes of $B$-subsections such that $b_i$ is uniquely determined by $x_i$ and has defect group $\C_D(x_i)$. Since $l(b_i)=1$, the Cartan matrix of $b_i$ is $(|\C_D(x_i)|)$ for $i=1,\ldots,s$. The elements $x_1,\ldots,x_s$ can be complemented to a set of representatives $x_1,\ldots,x_t$ for the conjugacy classes of $D$. The class equation for $D$ shows that
\[\gamma(B)=\sum_{i=1}^s\frac{1}{|\C_D(x_i)|}\le\sum_{i=1}^t\frac{1}{|\C_D(x_i)|}=1.\]
According to \autoref{con}, $\gamma(B)=1$ and $B$ is nilpotent.
\end{proof}

\begin{Thm}\label{conprop}
\autoref{con} holds in each of the following situations:
\begin{enumerate}[(i)]
\item $B$ is nilpotent.
\item\label{unique} $B$ is the only block of $G$.
\item $B$ has cyclic defect group.
\item $G$ is a symmetric group.
\item $G$ is a simple group of Lie type in defining characteristic.
\item $G$ is a quasisimple group appearing in the ATLAS.
\end{enumerate}
\end{Thm}
\begin{proof}\hfill
\begin{enumerate}[(i)]
\item This follows from the proof of \autoref{puig}. For the remaining parts we may assume that $B$ is non-nilpotent.

\item If $B$ is the only block of $G$, then $X(B)=X(G)=X$ and $X^\mathrm{t}\overline{X}=\diag(|\C_G(g_i)|:i=1,\ldots,k)$ where $g_1,\ldots,g_k$ represent the conjugacy classes of $G$. In particular, $e_i=|\C_G(g_i)|_p$ for $i=1,\ldots,k$. The class equation of $G$ reads 
\[\frac{1}{|\C_G(g_1)|}+\ldots+\frac{1}{|\C_G(g_k)|}=1.\]
Hence, $\gamma(B)>1$ unless $G$ is a $p$-group in which case $B$ is nilpotent.

\item It is well-known that a non-nilpotent block with cyclic defect group has elementary divisor $1$ (see \cite[Theorem~8.6]{habil}). This implies $\gamma(B)>1$ as explained before \autoref{puig}.

\item Let $B$ be a $p$-block of weight $w$ of the symmetric group $S_n$. It is well-known that $B$ is nilpotent if and only if $w=0$ or $(p,w)=(2,1)$. Thus, let $w\ge 1$. If $p$ is odd or $w$ is even, then $C_1$ has elementary divisor $1$ by a theorem of Olsson~\cite[Corollary~3.13]{OlssonLDGS}. Now let $p=2$ and $w=2k+1\ge 3$. Then the multiplicity of $2$ as an elementary divisor of $C_1$ is the number of partitions of $w$ with exactly one odd part (this can be extracted from \cite[Theorem~4.5]{BessenOlssonSpin}). Since there are at least two such partitions (namely $(w)$ and $(2k,1)$), it follows that $\gamma(B)>1$.

\item Apart from finitely many exceptions (like $^2F_4(2)'$) which are covered by \eqref{computer} below, we may assume that $G$ has only two blocks: the principal block $B$ and a block of defect $0$ containing the Steinberg character (see \cite[Section~8.5]{Humphreysbook}). Malle~\cite[Corollary~4.2]{Malleregss} has shown that there are at least two non-conjugate elements $g,h\in G$ such that $|\C_G(g)|_p=|\C_G(h)|_p=1$. One of them accounts for an elementary divisor $1$ of the Cartan matrix of $B$. Thus, $\gamma(B)>1$.

\item\label{computer} In order to check the claim by computer, we replace $B$ by the union of its Galois conjugate blocks so that $X(B)^\mathrm{t}\overline{X(B)}$ becomes an integral matrix. The $p$-parts of the elementary divisors of that matrix can be computed efficiently with Frank Lübeck's \texttt{edim} package~\cite{edim} for GAP~\cite{GAP48}. Since Galois conjugate blocks clearly have the same fusion number, we need to divide by the number of Galois conjugate blocks in the end. 
It turns out that for all blocks $B$ of quasisimple groups in the ATLAS, $\gamma(B)>1$ unless all characters have the same degree. In the latter case, $B$ is nilpotent by \cite{OkuyamaTsushima} (in fact, An and Eaton have shown that all nilpotent blocks of quasisimple groups have abelian defect groups). \qedhere
\end{enumerate}
\end{proof}

We have also compared \autoref{con} to the Malle--Navarro Conjecture (mentioned in the introduction) for small groups ($|G|\le 2000$) without finding any differences. 

In the remainder of the paper we offer two reduction theorems. To this end, we review Olsson's work~\cite{OlssonLDG} on lower defect groups which makes use of the algebraically closed field $F:=\mathcal{O}/\J(\mathcal{O})$ of characteristic $p$. We denote the set of blocks of $G$ by $\Bl(G)$ and the set of conjugacy classes by $\Cl(G)$. For $B\in\Bl(G)$ let $\epsilon_B$ be the block idempotent of $B$ as a subalgebra of $FG$. Moreover, $K^+:=\sum_{x\in K}x\subseteq\Z(FG)$ is the class sum of $K\in\Cl(G)$.

\begin{Prop}\label{blcksp}
The set $\Cl(G)$ can be partitioned into a so-called \emph{block splitting} 
\[\Cl(G)=\bigcup_{B\in\Bl(G)}\Cl(B)\]
such that $\{K^+\epsilon_B:K\in\Cl(B)\}$ is a basis of $\Z(B)\subseteq\Z(FG)$ for every $B\in\Bl(G)$. If $x_K\in K\in\Cl(B)$, then the Sylow $p$-subgroups of $\C_G(x_K)$ are called \emph{lower defect groups} of $B$. Their orders are the elementary divisors of $B$, i.\,e. $\{e_1,\ldots,e_k\}=\{|\C_G(x_K)|_p:K\in\Cl(B)\}$ as multisets.
\end{Prop}
\begin{proof}
The existence of block splittings is proved in \cite[Proposition~2.2]{OlssonLDG} (the proof is revisited in the following lemma). To verify the second claim, we freely use the notation from \cite{OlssonLDG}. In particular, $m_B(P)$ denotes the multiplicity of a $p$-subgroup $P\in\mathcal{P}(G)$ as a lower defect group of $B$. By combining Theorems~3.2, 5.4(1) and Corollary~7.7 of \cite{OlssonLDG}, the multiplicity of $p^n$ in the multiset $\{|\C_G(x_K)|_p:K\in\Cl(B)\}$ is 
\[\sum_{\substack{P\in\mathcal{P}(G)\\|P|=p^n}}m_B(P)=\sum_{\substack{P\in\mathcal{P}(G)\\|P|=p^n}}\sum_{x\in \Pi(G)}m_B^{(x)}(P)=
\sum_{x\in \Pi(G)}\sum_{\substack{b\in\Bl(\C_G(x))\\b^G=B}}\sum_{\substack{Q\in\mathcal{P}(\C_G(x))\\|Q|=p^n}}m_b^{(1)}(Q).\]
Now by \cite[Remark on p.~285]{OlssonLDG}, 
\[\sum_{\substack{Q\in\mathcal{P}(\C_G(x))\\|Q|=p^n}}m_b^{(1)}(Q)\] 
is the multiplicity of $p^n$ as an elementary divisor of the Cartan matrix of $b$. Moreover, every $B$-subsection $(x,b)$ appears (up to $G$-conjugation) just once in the sum. 
\end{proof}

In the language of block splittings our conjecture can be rephrased as
\[\sum_{K\in\Cl(B)}|K|_p\ge |G|_p\]
with equality if and only if $B$ is nilpotent.

\begin{Lem}\label{blckZ}
Let $Z\le\Z(G)$ be of order $p$. Then there exists a block splitting $\Cl(G)=\bigcup_{B\in\Bl(G)}\Cl(B)$ such that $\Cl(B)=\{Kz:K\in\Cl(B)\}$ for all $B\in\Bl(G)$ and $z\in Z$.
\end{Lem}
\begin{proof}
In order to exploit Olsson's proof of the existence of block splittings, we recall the full details. Instead of the generalized Laplace expansion we make use of the Leibniz formula for determinants.
Let $\Bl(G)=\{B_1,\ldots,B_n\}$ and $\Cl(G)=\{K_1,\ldots,K_k\}$. Let $I_1\cup\ldots\cup I_n=\{1,\ldots,k\}$ be a partition such that $\{b_i:i\in I_j\}$ is an $F$-basis of $\Z(B_j)$. Then $b_1,\ldots,b_k$ is a basis of $\Z(B_1)\oplus\ldots\oplus\Z(B_n)=\Z(FG)$. On the other hand, the class sums $K_1^+,\ldots,K_k^+$ also form a basis of $\Z(FG)$. Hence, there exists an invertible matrix $A=(a_{ij})\in F^{k\times k}$ such that 
\begin{equation}\label{eqbl}
A\begin{pmatrix}
b_1\\
\vdots\\
b_k
\end{pmatrix}=\begin{pmatrix}
K_1^+\\
\vdots\\
K_k^+
\end{pmatrix}.
\end{equation}
Let $J=(J_1,\ldots,J_n)$ be a partition of $\{1,\ldots,k\}$ such that $|J_i|=|I_i|$ for $i=1,\ldots,n$. 
Let $\sigma_J\in S_n$ be the unique permutation which sends $J_i$ to $I_i$ for $i=1,\ldots,n$ and preserves the natural order of those sets. By the Leibniz formula,
\begin{align}
\notag0\ne\det A&=\sum_{\alpha\in S_k}\sgn(\alpha)\prod_{i=1}^ka_{i,\alpha(i)}=\sum_{J}\sum_{\alpha\in S_{I_1}\times\ldots\times S_{I_n}}\sgn(\alpha\sigma_J)\prod_{i=1}^ka_{i,\alpha\sigma_J(i)}\\
\notag&=\sum_{J}\sgn(\sigma_J)\Bigl(\sum_{\alpha_1\in S_{I_1}}\sgn(\alpha_1)\prod_{j\in J_1}a_{j,\alpha_1\sigma_J(j)}\Bigr)\ldots\Bigl(\sum_{\alpha_n\in S_{I_n}}\sgn(\alpha_n)\prod_{j\in J_n}a_{j,\alpha_n\sigma_J(j)}\Bigr)\\
&=\sum_{J}\sgn(\sigma_J)\det(A_{J_1I_1})\ldots\det(A_{J_nI_n})\label{finaleq}
\end{align}
where $A_{I_iJ_i}:=(a_{st}:s\in I_i,t\in J_i)$.
Hence, there exists some partition $J$ such that $\det(A_{J_iI_i})\ne 0$ for $i=1,\ldots,n$. We now multiply \eqref{eqbl} with the block idempotent $\epsilon_s$ of $B_s$ to get
\[A_{J_sI_s}(b_i:i\in I_s)=(K_j^+\epsilon_s:j\in J_s)\]
(notice that $b_j\epsilon_s=0$ if $j\notin I_s$). Hence, the sets $\Cl(B_s):=\{K_j:j\in J_s\}$ form a block splitting of $G$. 

Next we observe that $Z=\langle z\rangle$ acts by multiplication on $\Cl(G)$. In this way, $z$ induces a permutation $\pi_z$ on $\{1,\ldots,k\}$ such that $K_jz=K_{\pi_z(j)}$. Let $J'_i:=\pi_z(J_i)$ for $i=1,\ldots,n$. We claim that $J'$ makes the same contribution to \eqref{finaleq} as $J$. Since $\Z(B_s)$ is an ideal of $\Z(FG)$, we see that $(K_j^+z\epsilon_s:j\in J_s)$ is a basis of $\Z(B_s)z=\Z(B_s)$. Thus, there exists $A_s\in\GL(k(B_s),F)$ such that 
\[A_s(K_j^+\epsilon_s:j\in J_s)=(K_j^+z\epsilon_s:j\in J_s).\] 
Since $z^p=1$, it follows that $A_s^p=1$. In particular $\det(A_s)=1$, since $F$ has characteristic $p$. Let $\tau_s\in S_{J_s'}$ such that elements $\tau_s\pi_z(j)$ with $j\in J_s$ appear in their natural order. Let $P_s$ be the permutation matrix corresponding to $\tau_s$. Then $A_{J_s'I_s}=P_sA_sA_{J_sI_s}$ for $s=1,\ldots,n$. In particular, $\det(A_{J_s'I_s})=\sgn(\tau_s)\det(A_{J_sI_s})$ for $s=1,\ldots,n$. If $p=2$, it is now clear that $J$ and $J'$ make the same contribution to \eqref{finaleq}. Thus, let $p>2$. Then $\pi_z$ has order $p$ and therefore $\sgn(\pi_z)=1$. 
Moreover, $\sigma_{J'}\tau_1\ldots\tau_s\pi_z=\sigma_J$. Consequently,
\[\sgn(\sigma_{J'})\det(A_{J_1'I_1})\ldots\det(A_{J_n'I_n})=\sgn(\sigma_J)\det(A_{J_1I_1})\ldots\det(A_{J_nI_n})\]
as desired. 

If $J'\ne J$, then the orbit of $J$ under $Z$ has length $p$. The corresponding $p$ equal summands of \eqref{finaleq} cancel out. Since we still have $\det A\ne 0$, there must exist a block splitting $J$ such that $J'=J$. The claim follows.
\end{proof}

In the remark after \cite[Propsosition~7.8]{OlssonLDG} Olsson states that there is no relation between a lower defect group of a block and its dominated block modulo a central $p$-subgroup. Nevertheless, we show that there is a relation if one considers all lower defect groups at the same time.

\begin{Prop}\label{reduction}
Let $Z$ be a $p$-subgroup of $\Z(G)$. Let $B$ be a $p$-block of $G$ and let $\overline{B}$ be the unique block of $G/Z$ dominated by $B$. Then $\gamma(B)=\gamma(\overline{B})$ and $B$ is nilpotent if and only if $\overline{B}$ is.
\end{Prop}
\begin{proof}
The second claim was proved in \cite[Lemma~2]{WatanabeNil}. A modern proof in terms of fusion systems can be given along the following lines. The fusion system $\mathcal{F}$ of $B$ contains $Z$ in its center, i.\,e. $\mathcal{F}=\C_{\mathcal{F}}(Z)$. One then shows that $\overline{\mathcal{F}}:=\mathcal{F}/Z$ is the fusion system of $\overline{B}$ (see \cite[Definition~5.9]{CravenBook}). Now by \cite[Proposition~5.60]{CravenBook}, there is a one-to-one correspondence between $\mathcal{F}$-essential subgroups and $\overline{\mathcal{F}}$-essential subgroups. Hence, $B$ is nilpotent if and only if $\overline{B}$ is.

To prove the first claim, we may assume that $|Z|=p$ by induction on $|Z|$. It is convenient to prove the claim for all blocks $B\in\Bl(G)$ at the same time.  
By \autoref{blckZ}, there exists a block splitting $\Cl(G)=\bigcup_{B\in\Bl(G)}\Cl(B)$ such that
\begin{equation}\label{KKz}
K\in\Cl(B)\Longleftrightarrow Kz\in\Cl(B)
\end{equation}
for all $K\in\Cl(G)$ and $z\in Z$.

The canonical epimorphism $G\to \overline{G}:=G/Z$ maps $\epsilon_B$ to $\overline{\epsilon_B}=\epsilon_{\overline{B}}$ since $\overline{B}$ is the only block dominated by $B$ (see \cite[p.~198]{Navarro}). Moreover, $\overline{K}\in\Cl(\overline{G})$ for every $K\in\Cl(G)$. Hence, $\{\overline{K}^+\epsilon_{\overline{B}}:K\in\Cl(B)\}$ spans $\Z(\overline{B})$. If $K,L\in\Cl(G)$ induce the same class $\overline{K}=\overline{L}$, then $L=Kz$ for some $z\in Z$. In this case \eqref{KKz} implies that $K\in\Cl(B)\Longleftrightarrow L\in\Cl(B)$.
Thus, after removing duplicates from the set $\Cl(\overline{B}):=\{\overline{K}:K\in\Cl(B)\}$, we obtain a partition
\[\Cl(\overline{G})=\bigcup_{\overline{B}\in\Bl(\overline{G})}\Cl(\overline{B}).\]
Since \[|\Cl(\overline{G})|=\dim\Z(F\overline{G})=\sum_{\overline{B}\in\Bl(\overline{G})}\dim\Z(\overline{B})\le\sum_{\overline{B}\in\Bl(\overline{G})}|\Cl(\overline{B})|=|\Cl(\overline{G})|,\]
the sets $\Cl(\overline{B})$ form a block splitting of $\Cl(\overline{G})$.

Finally, we determine the elementary divisors $e_1,\ldots,e_k$ of $B$. 
By \autoref{blcksp}, we may label $\Cl(B)=\{K_1,\ldots,K_k\}$ such that $e_i|K_i|_p=|G|_p$ for $i=1,\ldots,k$. If $|\overline{K_i}|=|K_i|$, then the $p$ classes $K_iz\in\Cl(B)$ with $z\in Z$ are all distinct. Since $\frac{1}{p}e_i|\overline{K_i}|_p=|\overline{G}|_p$, we have $p$ of the $e_i$, say $e_{i_1}=\ldots=e_{i_p}$ accounting for one elementary divisor $\overline{e}_i:=\frac{1}{p}e_i$ of $\overline{B}$. If, on the other hand, $|K_i|=p|\overline{K_i}|$, then $K_i=K_iz$ for all $z\in Z$. In this case we set $\overline{e}_i:=e_i$. This gives the elementary divisors $\overline{e}_1,\ldots,\overline{e}_l$ of $\overline{B}$
such that
\[\gamma(\overline{B})=\frac{1}{\overline{e}_1}+\ldots+\frac{1}{\overline{e}_l}=\frac{1}{e_1}+\ldots+\frac{1}{e_k}=\gamma(B).\qedhere\]
\end{proof}

Our final result is a reduction for blocks of $p$-solvable groups to a purely group-theoretic assertion (it might be called a \emph{projective} class equation).

\begin{Prop}\label{psolv}
\autoref{con} holds for all $p$-blocks of $p$-solvable groups if and only if the following is true: Let $G$ be a $p$-solvable group such that $Z:=\Z(G)=\pcore_{p'}(G)\le G'$ is cyclic and $Z\ne\pcore^p(G)$. Let $K_1,\ldots,K_n$ be the conjugacy classes of $G/Z$ consisting of elements $xZ$ such that $\C_{G/Z}(xZ)=\C_G(x)/Z$. Then $|K_1|_p+\ldots+|K_n|_p>|G|_p$.
\end{Prop}
\begin{proof}
Let $B$ be a $p$-block of a $p$-solvable group $G$. 
By Broué~\cite[Théorème~5.5]{Broueiso}, $B$ is isotypic to a block of a $p$-solvable group $H$ such that $\pcore_{p'}(H)\subseteq\Z(H)$. Since isotypies preserve the generalized decomposition matrices up to basic sets (see \cite[Théorème~4.8]{Broueiso}), also the elementary divisors of $B$ are preserved. Hence, 
we may assume that $Z:=\pcore_{p'}(G)\le\Z(G)$. By \autoref{reduction}, we may further assume that $Z=\Z(G)$. 
Recall that $\Ker(B)\le Z$ by \cite[Theorem~6.10]{Navarro}. Obviously, $B$ is nilpotent if and only if the isomorphic block $\overline{B}$ of $G/\Ker(B)$ is nilpotent. Moreover, $X(\overline{B})$ is obtained from $X(B)$ by removing duplicate columns. It follows that $\gamma(B)=\gamma(\overline{B})$. 
By replacing $G$ with $G/\Ker(B)$, we may assume that $B$ is faithful. By \autoref{conprop}, we may assume that $B$ is non-nilpotent and therefore $Z\ne\pcore^p(G)$. The reduction to $Z\le G'$ will be established at the end of the proof.

In order to construct a block splitting, we need to consider all blocks of $G$. 
By \cite[Theorem~10.20]{Navarro}, the blocks of $G$ can be labeled by $\lambda\in\Irr(Z)$ such that $\Irr(B_\lambda)=\Irr(G|\lambda)$. 
The block idempotent of $B_\lambda$ is just the ordinary character idempotent $\epsilon_\lambda\in\Z(FZ)$ (see \cite[p.~51]{Navarro}). Note that $Z$ acts by multiplication on $\Cl(G)$. Let $Z_K:=\{z\in Z:zK=K\}\le Z$ be the stabilizer of $K\in\Cl(G)$. The classes in the orbit of $K$ can be labeled arbitrarily by $\Irr(Z/Z_K)\le\Irr(Z)$, say $\{Kz:z\in Z\}=\{K_\lambda:\lambda\in\Irr(Z/Z_K)\}$. We define
\[\Cl(B_\lambda):=\{K_\lambda:K\in\Cl(G),Z_K\subseteq \Ker(\lambda)\}.\]
Note that for $K\in\Cl(B_\lambda)$ and $z\in Z$ we have
\[(Kz)^+\epsilon_\lambda=\lambda(z)K^+\epsilon_\lambda\in F\cdot K^+\epsilon_\lambda.\]
On the other hand, if $z\in Z_K\setminus\Ker(\lambda)$, then 
\[K^+\epsilon_\lambda=(Kz)^+\epsilon_\lambda=\lambda(z)K^+\epsilon_\lambda=0.\]
It follows easily that $\Z(B_\lambda)$ is spanned by $\{K^+\epsilon_\lambda:K\in\Cl(B_\lambda)\}$. Since 
\[|\Cl(G)|=\dim \Z(FG)=\sum_{\lambda\in\Irr(Z)}\dim\Z(B_\lambda)\le\sum_{\lambda\in\Irr(Z)}|\Cl(B_\lambda)|=|\Cl(G)|,\]
we conclude that $\Cl(G)=\bigcup_{\lambda\in\Irr(Z)}\Cl(B_\lambda)$ is indeed a block splitting of $G$ (this can also be explained with the notion of \emph{good} conjugacy classes in \cite[Theorem~5.14]{Navarro2}). 

We only need to verify the claim for a faithful block $B=B_\lambda$, i.\,e. $\Ker(\lambda)=1$ and $Z$ is cyclic. Here the conjugacy classes $K\in\Cl(B)$ represent the regular orbits of $Z$ on $\Cl(G)$. Thus, for $x\in K$ we have $\C_G(x)/Z=\C_{G/Z}(xZ)$ as desired. 
Now we fix coset representatives $\widehat{g}$ for every $g\in G/Z$. 
Then the equation $\widehat{g}\widehat{h}=\alpha(g,h)\widehat{gh}$ where $g,h\in G/Z$ defines a 2-cocycle $\alpha\in\Z^2(G/Z,Z)$. Let $\beta:=\lambda\circ\alpha\in\Z^2(G/Z,F^\times)$. It is well-known that the map $g\mapsto\widehat{g}\epsilon_\lambda$ induces an algebra isomorphism between the twisted group algebra $F_\beta[G/Z]$ and $B$. The class sums $K^+\epsilon$ with $K\in\Cl(B)$ correspond to the so-called $\beta$-\emph{regular} class sums of $F_\beta[G/Z]$ (these are the only non-vanishing class sums in $F_\beta[G/Z]$ and therefore form a basis of $\Z(F_\beta[G/Z])$). 

Since $\beta$ can be regarded as an element of the Schur multiplier $\cohom^2(G/Z,F^\times)$, $F_\beta[G/Z]$ is also isomorphic to a faithful block of a covering group $\widetilde{G}$ with cyclic $\widetilde{Z}\le\Z(\widetilde{G})\cap\widetilde{G}'$ such that $\widetilde{G}/\widetilde{Z}\cong G/Z$. 
Again the $\beta$-regular class sums correspond to the regular orbits of $\widetilde{Z}$ on $\Cl(\widetilde{G})$. Moreover, we still have $\widetilde{Z}=\pcore_{p'}(\widetilde{G})\ne\pcore^p(\widetilde{G})$. Hence, we may replace $G$ by $\widetilde{G}$ and $Z$ by $\widetilde{Z}$. Since $B$ is non-nilpotent, it remains to show that 
\[\sum_{K\in\Cl(B)}|K|_p=|G|_p\gamma(B)>|G|_p.\qedhere\]
\end{proof}

A concrete example to \autoref{psolv} is the double cover of $S_3\times S_3$ for $p=3$. Here the fusion number of the unique non-principal block is $10/9$ (this is the smallest number larger than $1$ that we have encountered).

If \autoref{con} can be verified for blocks of $p$-solvable groups, then it also holds for blocks with normal defect groups since such blocks are splendid Morita equivalent to blocks of $p$-solvable groups by Külshammer~\cite{Kuelshammer}. Similarly, \autoref{con} would follow for blocks with abelian defect groups if additionally Broué's Conjecture is true.

\section*{Acknowledgment}
I thank Gabriel Navarro for stimulating discussions on this paper, Christine Bessenrodt for making me aware of \cite{BessenOlssonSpin} and Gunter Malle for providing \cite{Malleregss}. Moreover, I appreciate a very careful reading of an anonymous referee.
The work is supported by the German Research Foundation (\mbox{SA 2864/1-2} and \mbox{SA 2864/3-1}).

\end{document}